\documentclass[11pt]{article}
\usepackage{latexsym, amsmath, amssymb, amsthm, a4, epsfig}
\usepackage{empheq}
\usepackage{graphicx}
\usepackage{color}
\usepackage{comment}
\usepackage{tikz}
\usepackage{url}

\newtheorem{theorem}{Theorem}[section]
\newtheorem{cor}[theorem]{Corollary}
\newtheorem{lemma}[theorem]{Lemma}
\newtheorem{prop}[theorem]{Proposition}

\theoremstyle{definition}
\newtheorem{remark}{Remark}
\newtheorem{example}{Example}

\newcommand{\p}{\partial}

\newcommand{\Nbb}{\mathbb{N}}

\newcommand{\Rbb}{\mathbb{R}}

\renewcommand{\div}{\mbox{div}\,}

\newcommand{\la}{\langle}
\newcommand{\ra}{\rangle}

\newcommand{\Gb}{\beta}

\newcommand{\Ge}{\epsilon}

\newcommand{\Gf}{\Gvf}
\newcommand{\Gvf}{\varphi}
\newcommand{\Gg}{\gamma}
\newcommand{\Gc}{\chi}

\newcommand{\Gm}{\mu}

\newcommand{\Gs}{\sigma}

\newcommand{\Gy}{\psi}
\newcommand{\Gz}{\zeta}

\newcommand{\GG}{\Gamma}

\newcommand{\GO}{\Omega}

\newcommand{\beq}{\begin{equation}}
\newcommand{\eeq}{\end{equation}}

\def\ol{\overline}

\def\sm{\setminus}

\DeclareMathOperator{\divergence}{div}

\DeclareMathOperator{\supp}{supp}
\DeclareMathOperator*{\essinf}{ess\,inf}

\DeclareMathOperator{\tr}{\mathrm{tr}}
\DeclareMathOperator{\Lip}{Lip}
\DeclareMathOperator{\dist}{dist}
\DeclareMathOperator{\loc}{\mathrm{loc}}
\DeclareMathOperator{\Dom}{Dom}

\numberwithin{equation}{section}
\numberwithin{figure}{section}

\allowdisplaybreaks

\newcommand{\normal}{n}
\newcommand{\adv}{\Gb}
\newcommand{\reac}{\Gm}

\newcommand{\weight}{\adv \cdot \normal}

\begin{document}

\title{A revisit on well-posedness of a boundary value problem of a stationary advection equation without the separation condition}

\author{Masaki Imagawa\thanks{Graduate School of Informatics, Kyoto University (m\_imagawa@acs.i.kyoto-u.ac.jp, d.kawagoe@acs.i.kyoto-u.ac.jp)} and Daisuke Kawagoe\footnotemark[1]}

\date{\today}

\maketitle

\begin{abstract}
We consider a boundary value problem of a stationary advection equation in a bounded domain with Lipschitz boundary. It is known to be well-posed in $L^p$-based function spaces for $1 < p < \infty$ under the separation condition of the inflow and the outflow boundaries. In this article, we provide another sufficient condition for the well-posedness with $1 \leq p \leq \infty$.

{\flushleft{{\bf Keywords:} first order PDE, boundary value problem, trace operators}}
{\flushleft{{\bf 2020 Mathematics Subject Classification.} 35F15, 35L02}}

\end{abstract}

\section{Introduction}

In this article, we discuss well-posedness of a boundary value problem of a stationary advection equation in a bounded domain with Lipschitz boundary in $L^p$-based function spaces for $1 \leq p \leq \infty$ without the separation condition.

Let $\GO$ be a bounded domain in $\Rbb^d$ with Lipschitz boundary. Also, let $\adv \in W^{1, \infty}(\GO)^d$ and $\reac \in L^\infty(\GO)$. Let $1 \leq p \leq \infty$ and let $q$ be the H\"older conjugate, namely, a real number such that $p^{-1} + q^{-1} = 1$. Here and in what follows, we use the convention: $p^{-1} = 0$ when $p = \infty$. 

We consider the following boundary value problem:
\beq \label{BVP}
\begin{cases}
\adv \cdot \nabla u + \reac u = f &\mbox{ in } \GO,\\
u = g &\mbox{ on } \GG_-,
\end{cases}
\eeq
where $f \in W^q_{\adv, +}(\GO)'$ and $g \in L^p(\GG_-; |\weight|)$. The boundary value problem \eqref{BVP} is a basic model of inviscid flow. The sets $\GG_-$ and $\GG_+$ are called the inflow and the outflow boundaries respectively, and they are characterized by the sign of the inner product $\adv \cdot \normal$, where $n$ is the outward unit normal vector to $\p \GO$. We give their precise definitions in Section \ref{sec:pre}. In what follows, we assume that the inflow boundary $\GG_-$ has the positive measure with respect to the surface measure. 

For $1 \leq p \leq \infty$, a Banach space $W^p_\adv(\GO)$ is defined by
\[
W^p_\adv(\GO) := \{ v \in L^p(\GO) \mid \adv \cdot \nabla v \in L^p(\GO) \}
\]
with the norm $\| \cdot \|_{W^p_\adv(\GO)}$ defined by
\[
\| u \|_{W^p_\adv(\GO)} := \left( \| u \|_{L^p(\GO)}^p + \| \adv \cdot \nabla u \|_{L^p(\GO)}^p \right)^{\frac{1}{p}}
\]
for $1 \leq p < \infty$ and
\[
\| u \|_{W^\infty_\adv(\GO)} := \max \left\{ \| u \|_{L^\infty(\GO)}, \| \adv \cdot \nabla u \|_{L^\infty(\GO)} \right\}.
\]
The norms $\| \cdot \|_{L^p(\GG_\pm; |\weight|)}$ on function spaces $L^p(\GG_\pm; |\weight|)$ are defined by
\[
\| g \|_{L^p(\GG_\pm; |\weight|)} := \left( \int_{\GG_\pm} |g|^p |\weight|\,d\Gs_x \right)^{\frac{1}{p}},
\]
where $d\Gs_x$ is the surface measure on $\p \GO$. For a Banach space $X$, let $X'$ denote its dual space. Namely, $W^q_{\adv, +}(\GO)'$ in the setting of the boundary value problem \eqref{BVP} is the dual space of $W^q_{\adv, +}(\GO)$. Here, $W^q_{\adv, +}(\GO)$ is a subspace of $W^q_\adv(\GO)$ and its definition will be given later. We note that the dual space $X'$ is equipped with a norm $\| \cdot \|_{X'}$ defined by
\[
\| f \|_{X'} := \sup_{u \in X \setminus \{ 0 \}} \frac{|f(u)|}{\| u \|_X}. 
\] 
It is worth mentioning that $C^\infty(\ol{\GO})$ is dense in $W^p_\adv(\GO)$ for $1 \leq p < \infty$ \cite{J}. Hence, $\Lip(\ol{\GO})$ is also dense there, where $\Lip(\ol{\GO})$ is the function space consisting of Lipschitz functions on $\ol{\GO}$. 

The boundary value problem \eqref{BVP} with $f \in L^2(\GO)$ and $g = 0$ was investigated in the context of the semigroup theory \cite{1970Ba, GL}, where elliptic regularization was applied. In the context of the numerical computation, it was regarded as one of Friedrichs' systems when $p = 2$ and $f \in L^2(\GO)$ \cite{2006EG}. However, because the traces $v|_{\GG_\pm}$ for a function $v \in W^p_\adv(\GO)$ may not belong to $L^p(\GG_\pm; |\weight|)$ in general \cite{1970Ba}, it is not straightforward to define boundary operators appearing in Friedrichs' systems. 

In order to avoid this difficulty, Ern and Guermond \cite{2006EG} introduced the separation condition, or the boundaries $\GG_\pm$ are well-separated:
\[
\dist (\GG_+, \GG_-) := \inf_{x \in \GG_+, x' \in \GG_-} |x - x'| > 0.
\]
Under the separation condition, the trace operators $\Gg_\pm: W^p_\adv(\GO) \to L^p(\GG_\pm; |\weight|)$ for $1 \leq p < \infty$ are bounded. More precisely, the operators $\Gg_\pm: W^p_\adv(\GO) \to L^p(\GG_\pm; |\weight|)$ with $\Dom(\Gg_\pm) = \Lip(\ol{\GO})$ have bounded extensions. Based on this idea, well-posedness results of the boundary value problem \eqref{BVP} for general $p$ with $1 < p < \infty$ have been established \cite{BH, C, MTZ}. However, the separation condition does not seem natural from the view points of mathematical analysis on boundary value problems and its numerical computations. Actually, it was not assumed in \cite{1970Ba, GL}. This discrepancy motivates us to remove the condition. We note that the well-posedness of the boundary value problem \eqref{BVP} for $g \in L^2_{\loc}(\GG_-; |\weight|)$ without the separation condition was mentioned in \cite[Example 2.12 (iii)]{2004EG} without a proof or a reference.

In this article, we consider the trace operators in two ways. One way is that they are  continuous operators from $W^p_\adv(\GO)$ to $L^p_{\loc}(\GG_\pm; |\weight|)$ for $1 \leq p < \infty$, which was mentioned in \cite{AP, 1970Ba, GL} and was introduced for the transport equation \cite{DL}. By taking closed subsets $\GG_\pm'$ in $\GG_\pm$, we can say that the trace operators $\tilde{\Gg}_\pm: W^p_\adv(\GO) \to L^p(\GG_\pm'; |\weight|)$ with $\Dom(\tilde{\Gg}_\pm) = \Lip(\ol{\GO})$ have bounded extensions. However, they do not have bounded extensions as operators from $W^p_\adv(\GO)$ to $L^p(\GG_\pm; |\weight|)$, which causes difficulty for discussing the well-posedness. The other way is that they are closed operators from $W^p_\adv(\GO)$ to $L^p(\GG_\pm; |\weight|)$ with $\Dom(\Gg_\pm) = W^p_{\adv, \tr}(\GO)$, where
\[
W^p_{\adv, \tr}(\GO) := \{ u \in W^p_\adv(\GO) \mid \tilde{\Gg}_\pm u \in L^p(\GG_\pm; |\weight|) \}.
\]
We introduce the graph norm $\| \cdot \|_{W^p_{\adv, \tr}(\GO)}$ on the space $W^p_{\adv, \tr}(\GO)$;
\[
\| u \|_{W^p_{\adv, \tr}(\GO)} := \left( \| u \|_{W^p_\adv(\GO)}^p + \| \tilde{\Gg}_+ u \|_{L^p(\GG_+; \weight)}^p + \| \tilde{\Gg}_- u \|_{L^p(\GG_-; |\weight|)}^p \right)^{\frac{1}{p}}.
\]
Then, by Clarkson's first and second inequalities, the space $W^p_{\adv, \tr}(\GO)$ with the norm $\| \cdot \|_{W^p_{\adv, \tr}(\GO)}$ is a reflexive Banach spaces for any $1 < p < \infty$. We notice that $\Gg_\pm u = \tilde{\Gg}_\pm u$ for all $u \in W^p_{\adv, \tr}(\GO)$. Furthermore, we introduce subspaces $W^p_{\adv, \pm}(\GO)$ of $W^p_{\adv, \tr}(\GO)$, which are defined by
\[
W^p_{\adv, \pm}(\GO) := \{ u \in W^p_{\adv, \tr}(\GO) \mid \tilde{\Gg}_\pm u = 0 \}.
\]

In this article, we restrict ourselves to working on the space $W^p_{\adv, \tr}(\GO)$ rather than $W^p_\adv(\GO)$. However, as was mentioned in \cite{R}, it is not obvious whether Green's formula, which is a key tool in our argument, holds there or not. In order to justify our analysis, we require that
\begin{equation} \label{eq:density}
W^p_{\adv, \tr}(\GO) = \ol{\Lip(\ol{\GO})}^{\| \cdot \|_{W^p_{\adv, \tr}(\GO)}}
\end{equation}
for all $1 \leq p < \infty$. It is known that, if the $(d-2)$-dimensional Hausdorff measure of the boundary $\p \GG_-$ is finite, then $H^1_{\GG_-}(\GO)$ is dense in $W^2_{\adv, -}(\GO)$ \cite{1970Ba}, where
\[
H^1_{\GG_-}(\GO) := \{ u \in H^1(\GO) \mid u|_{\GG_-} = 0 \}.
\]
With the same assumption, we can show that the density \eqref{eq:density} holds for $1 \leq p \leq 2$. However, it is not obvious whether it still holds for $2 < p < \infty$. We shall give a sufficient condition for the relation \eqref{eq:density} to hold for all $1 \leq p < \infty$.

We first discuss well-posedness of the boundary value problem \eqref{BVP} in the weak sense under the assumption \eqref{eq:density}. A weak solution $u \in L^p(\GO)$ to the boundary value problem \eqref{BVP} is defined to be a function which satisfies
\beq \label{BVP_weak}
\int_\GO u (-\nabla \cdot (\adv v) + \reac v)\,dx = \la f, v \ra_{W^q_{\adv, +}(\GO)', W^q_{\adv, +}(\GO)} + \int_{\GG_-} g v |\weight|\,d\Gs_x
\eeq
for all $v \in W^q_{\adv, +}(\GO)$. 

\begin{theorem} \label{thm:WP_Lp}
Let $1 < p \leq \infty$ and $q$ be its H\"older conjugate. Assume \eqref{eq:density} and 
\beq \label{reacp}
\reac_p := \essinf_{x \in \GO} \left( \reac(x) - \frac{1}{p} \divergence \adv(x) \right) > 0.
\eeq
Then, for any $f \in W^q_{\adv, +}(\GO)'$ and $g \in L^p(\GG_-; |\weight|)$, there exists a unique weak solution $u$ to the boundary value problem \eqref{BVP} in $L^p(\GO)$. Moreover, we have 
\[
\| u \|_{L^p(\GO)} \leq C_{1, p} \left( \| f \|_{W^q_{\adv, +}(\GO)'} + \| g \|_{L^p(\GG_-; |\weight|)} \right),
\]
where
\begin{align}
C_{1, p} :=& \left( 1 + C_{2, p} \right) \reac_p^{-1} \left ( 1 +  \reac_p + \| \adv \|_{W^{1, \infty}(\GO)} + \| \reac \|_{L^\infty(\GO)} \right), \label{def:C1p} \\
C_{2, p} :=& p^{\frac{1}{p}} + \| \adv \|_{W^{1, \infty}(\GO)}^{\frac{1}{p}}. \label{def:C2p}
\end{align}
\end{theorem}

We note that $\| \cdot \|_{L^\infty(\GG_\pm; |\weight|)} = \| \cdot \|_{L^\infty(\GG_\pm)}$. Also, constants $C_{1, \infty}$ in \eqref{def:C1p} and $C_{2, \infty}$ are regarded as
\begin{equation} \label{def:C1_infty}
C_{1, \infty} = 3 \reac_\infty^{-1} \left ( 1 +  \reac_\infty + \| \adv \|_{W^{1, \infty}(\GO)} + \| \reac \|_{L^\infty(\GO)} \right)
\end{equation}
and $C_{2, \infty} = 2$ respectively, which is obtained by taking the limit $p \to \infty$ in \eqref{def:C1p} and \eqref{def:C2p} formally.

\begin{remark}
According to \cite{C}, the assumption \eqref{reacp} can be replaced by the following one for $1 < p < \infty$: there exists a Lipschitz function $\Gz$ on $\ol{\GO}$ such that
\[
\essinf_{x \in \GO} e^{\Gz(x)} \left( \reac(x) - \frac{1}{p} \divergence \adv(x) -\frac{1}{p} \adv(x) \cdot \nabla \Gz(x) \right) > 0.
\]
The existence of the function $\Gz$ is known for some $C^1$ vector fields \cite{DEF}. However, it is not obvious whether this assumption covers the case $p=\infty$.
\end{remark}

We next discuss well-posedness of the boundary value problem \eqref{BVP} in the strong sense; a solution $u$ in $W^p_{\adv, \tr}(\GO)$. 

Let $f \in L^p(\GO)$ in the boundary value problem \eqref{BVP}. Then, we may regard it as a bounded linear functional on $W^q_{\adv, +}(\GO)$ by
\[
\la f, v \ra_{W^q_{\adv, +}(\GO)', W^q_{\adv, +}(\GO)} := \int_\GO f v\,dx, \quad v \in W^q_{\adv, +}(\GO).
\]
Thus, by Theorem \ref{thm:WP_Lp}, there exists the unique weak solution $u \in L^p(\GO)$ to the  boundary value problem \eqref{BVP}. Furthermore, we have the following regularity result.

\begin{theorem} \label{thm:WP_Wp}
Let $1 \leq p \leq \infty$, and assume \eqref{eq:density} and \eqref{reacp}. Then, for any $f \in L^p(\GO)$ and $g \in L^p(\GG_-; |\weight|)$, there exists a unique strong solution $u$ to the boundary value problem \eqref{BVP} in $W^p_{\adv, \tr}(\GO)$. Moreover, we have 
\begin{equation*}
\| u \|_{W^p_{\adv, \tr}(\GO)} \leq C_{1, p}' \left( \| f \|_{L^p(\GO)} + \| g \|_{L^p(\GG_-; |\weight|)} \right),
\end{equation*}
where $C_{1, p}' := (1 + C_{2, p}) \{ 2 + \left( 1 + \| \reac \|_{L^\infty(\GO)} \right) C_{1, p} \}$, and $C_{1, p}$ is the constant defined by \eqref{def:C1p}.
\end{theorem}

We remark that the traces of a function $u \in W^\infty_\adv(\GO)$ on $L^\infty_{\loc}(\GG_\pm)$ are not defined due to the lack of density of $\Lip(\ol{\GO})$. Thus, we did not define the space $W^\infty_{\adv, \tr}(\GO)$. However, since $L^\infty(\GG_\pm) \subset L^p(\GG_\pm; |\weight|)$ and $W^\infty_\adv(\GO) \subset W^p_\adv(\GO)$ for all $1 < p < \infty$, we may consider the local $L^p$ traces of $u \in W^\infty_\adv(\GO)$. Theorem \ref{thm:WP_Wp} for $p = \infty$ means that, if $f \in L^\infty(\GO)$ and $g \in L^\infty(\GG_-)$, then the boundary value problem \eqref{BVP} eventually has the unique solution $u \in W^\infty_\adv(\GO)$ for $1 \leq p < \infty$. In that case, the norm $\| \cdot \|_{W^\infty_{\adv. \tr}(\GO)}$ of the function space $W^\infty_{\adv, \tr}(\GO)$ is defined by
\[
\| u \|_{W^\infty_{\adv. \tr}(\GO)} := \max \left\{ \| u \|_{W^\infty_\adv(\GO)}, \| \tilde{\Gg}_+ u \|_{L^\infty(\GG_+)}, \| \tilde{\Gg}_- u \|_{L^\infty(\GG_-)} \| \right\}.
\]

Unlike Theorem \ref{thm:WP_Lp}, the case $p = 1$ is included in Theorem \ref{thm:WP_Wp}. Also, as a byproduct, Theorem \ref{thm:WP_Wp} implies the surjectivity of the trace operators $\Gg_\pm: W^p_{\adv, \tr}(\GO) \to L^p(\GG_\pm; |\weight|)$.

As corollaries, we obtain well-posedness results for the adjoint boundary value problem:
\begin{equation} \label{BVP_ad}
\begin{cases}
-\nabla \cdot (\adv u) + \reac u = f &\mbox{ in } \GO,\\
u = g &\mbox{ on } \GG_+,
\end{cases}
\end{equation}
where $f \in W^q_{\adv, -}(\GO)'$ and $g \in L^p(\GG_+; \weight)$.

\begin{cor} \label{cor:WP_Lp_ad}
Let $1 < p \leq \infty$ and $q$ be its H\"older conjugate. Assume \eqref{eq:density} and 
\begin{equation} \label{reacq}
\reac_q > 0,
\end{equation}
where $\reac_q$ is the constant defined by \eqref{reacp} with $p$ replaced by $q$. Then, for any $f \in W^q_{\adv, -}(\GO)'$ and $g \in L^p(\GG_+; \weight)$, there exists a unique weak solution $u$ to the boundary value problem \eqref{BVP_ad} in $L^p(\GO)$. Moreover, we have 
\[
\| u \|_{L^p(\GO)} \leq \tilde{C}_{1, q} \left( \| f \|_{W^q_{\adv, +}(\GO)'} + \| g \|_{L^p(\GG_+; |\weight|)} \right),
\]
where
\begin{equation} \label{def:C1q} 
\tilde{C}_{1, q} := \left( 1 + C_{2, q} \right) \reac_q^{-1} \left ( 1 +  \reac_q + \| \reac \|_{L^\infty(\GO)} \right). 
\end{equation}
\end{cor}

\begin{cor} \label{cor:WP_Wp_ad}
Let $1 \leq p \leq \infty$ and $q$ be its H\"older conjugate. Assume \eqref{eq:density} and \eqref{reacq}. Then, for any $f \in L^p(\GO)$ and $g \in L^p(\GG_+; \weight)$, there exists a unique strong solution $u$ to the boundary value problem \eqref{BVP} in $W^p_{\adv, \tr}(\GO)$. Moreover, we have 
\begin{equation*}
\| u \|_{W^p_{\adv, \tr}(\GO)} \leq \tilde{C}_{1, q}' \left( \| f \|_{L^p(\GO)} + \| g \|_{L^p(\GG_-; |\weight|)} \right),
\end{equation*}
where $\tilde{C}_{1, q}' := (1 + C_{2, q}) \{ 1 + \left( 1 + \| \reac \|_{L^\infty(\GO)} \right) \tilde{C}_{1, q}\}$, and $\tilde{C}_{1, q}$ is the constant defined by \eqref{def:C1q}. 
\end{cor}

\begin{remark}
If the surface measure of $\Gamma_-$ is $0$, then the boundary condition is not posed in \eqref{BVP}. Such kind of problems have been discussed extensively; for example, \cite{1986Veiga, GT}. In this case, the above theorems still hold with slight modifications. For example, Theorem \ref{thm:WP_Lp} is modified as follows: Let $1 < p \leq \infty$ and $q$ be its H\"older conjugate. Assume \eqref{reacp}. Then, for any $f \in W^q_{\adv, +}(\GO)'$, there exists a unique weak solution $u$ to the boundary value problem \eqref{BVP} in $L^p(\GO)$. Moreover, we have 
\[
\| u \|_{L^p(\GO)} \leq C_{1, p} \| f \|_{W^q_{\adv, +}(\GO)'},
\]
where $C_{1, p}$ is the constant defined by \eqref{def:C1p}. Theorem \ref{thm:WP_Wp} is modified in the same way.
\end{remark}

We briefly mention recent results on the standard Sobolev regularity of solutions to the boundary value problem \eqref{BVP} with $H^1$ divergence-free vector fields \cite{2017Be, SP} as relevant works.

The rest part of this article is organized as follows. In Section \ref{sec:pre}, we give precise definitions of $\GG_\pm$ and $\GG_0$. We also describe properties of the trace operators. In addition, we introduce the Banach-Ne\v{c}as-Babu\v{s}ka theorem, which plays the key role in our proofs. In Section \ref{sec:WP_Lp}, we discuss well-posedness of the boundary value problem \eqref{BVP} in $L^p(\GO)$ for $1 < p \leq \infty$. In Section \ref{sec:WP_Wp}, we further investigate well-posedness of the boundary value problem \eqref{BVP} in $W^p_\adv(\GO)$ for $1 \leq p \leq \infty$. In Section \ref{sec:sufficient}, we provide a sufficient condition for the relation \eqref{eq:density} which violates the separation condition.

\section{Preliminary} \label{sec:pre}

In this section, we give precise definitions of $\GG_\pm$ and $\GG_0$. We also describe properties of the trace operators. In addition, we introduce the Banach-Ne\v{c}as-Babu\v{s}ka theorem, which plays the key role in our proofs. 

We give definitions of $\GG_\pm$ and $\GG_0$, which are a modification of those in \cite{ABL}.

Since $\adv \in W^{1, \infty}(\GO)^d$, we regard it as a Lipschitz function on $\ol{\GO}$. Moreover, since the domain $\GO$ is bounded and the boundary is Lipschitz, there exists a Lipschitz function $\tilde{\adv}$ on $\Rbb^d$ such that $\tilde{\adv}(x) = \adv(x)$ for all $x \in \ol{\GO}$. For the existence of such an extension, see \cite{Le, Stein} for example.

We take $x_0 \in \p \GO$ and consider the following Cauchy problem:
\begin{equation} \label{Cauchy_boundary}
\begin{cases}
\dfrac{d}{dt} x(t) = \tilde{\adv}(x(t)),\\
x(0) = x_0.
\end{cases}
\end{equation}
Since $\tilde{\adv}$ is Lipschitz, the Cauchy problem \eqref{Cauchy_boundary} has a unique local solution. Let $x(t)$ be the unique solution. We decompose the boundary $\p \GO$ into three parts:
\begin{align*}
\GG_+ &:= \{ x_0 \in \p \GO \mid \mbox{There exists } T_+ > 0 \mbox{ such that } x(t) \notin \ol{\GO} \mbox{ for all } 0< t < T_+. \},\\
\GG_- &:= \{ x_0 \in \p \GO \mid \mbox{There exists } T_- > 0 \mbox{ such that } x(t) \in \GO \mbox{ for all } 0< t < T_-. \},\\
\GG_0 &:= \p \GO \setminus (\GG_+ \cup \GG_-).
\end{align*}
We call these subsets the outflow boundary, the inflow boundary, and the characteristic boundary respectively. By considering the normal component of the dynamics \eqref{Cauchy_boundary}, we see that $\pm \adv(x_0) \cdot n(x_0) > 0$ for a.e. $x_0 \in \GG_\pm$, where $n(x_0)$ is the outward unit normal vector at $x_0 \in \p \GO$. Also, due to continuous dependence of the solution to the Cauchy problem on initial values and due to the assumption that $\GO$ and $\Rbb^d \setminus \ol{\GO}$ are open, we see that $\GG_\pm$ are open. Therefore, our definitions of $\GG_\pm$ and $\GG_0$ are independent of choice of the extension $\tilde{\adv}$.

We proceed to discuss properties of the trace operators. We start from the following estimate.

\begin{lemma} \label{lem:trace_adv} 
Let $1 \leq p < \infty$. Then, we have
\begin{equation} \label{est:Green}
\| u \|_{L^p(\GG_\pm; |\weight|)} \leq C_{2, p} \| u \|_{W^p_\adv(\GO)} + \| u \|_{L^p(\GG_\mp; |\weight|)}
\end{equation}
for all $u \in \Lip(\ol{\GO})$, where $C_{2, p}$ is the constant defined by \eqref{def:C2p}. 
\end{lemma}

\begin{proof}
By Green's formula, we have
\begin{equation*}
\int_\GO (\adv \cdot \nabla u) |u|^{p-2} u\,dx = \frac{1}{p} \int_{\GG_+} |u|^p \weight\,d\Gs_x - \frac{1}{p} \int_{\GG_-} |u|^p |\weight|\,d\Gs_x - \frac{1}{p} \int_\GO (\div \adv) |u|^p\,dx
\end{equation*}
for all $u \in \Lip(\ol{\GO})$. Thus, we obtain
\begin{align*}
\int_{\GG_\pm} |u|^p |\weight|\,d\Gs_x \leq p \int_\GO |\adv \cdot \nabla u| |u|^{p-1}\,dx  + \int_\GO |\div \adv| |u|^p\,dx + \int_{\GG_\mp} |u|^p |\weight|\,d\Gs_x. 
\end{align*}

For the first term of the right hand side, by the H\"older inequality and Young's inequality, we have
\[
p \int_\GO |\adv \cdot \nabla u| |u|^{p - 1}\,dx \leq p \| \adv \cdot \nabla u \|_{L^p(\GO)} \| u \|_{L^p(\GO)}^{p-1} \leq \| \adv \cdot \nabla u \|^p + (p - 1) \| u \|_{L^p(\GO)}^{p-1}.
\]
For the second term of the right hand side, we get 
\[
\int_\GO |\div \adv| |u|^p\,dx \leq \| \adv \|_{W^{1, \infty}(\GO)} \| u \|_{L^p(\GO)}^p. 
\]
Thus, we have
\begin{align*}
\| u \|_{L^p(\GG_\pm; |\weight|)}^p \leq& \| \adv \cdot \nabla u \|_{L^p(\GO)}^p + (p - 1) \| u \|_{L^p(\GO)}^p + \| \adv \|_{W^{1, \infty}(\GO)} \| u \|_{L^p(\GO)}^p + \| u \|_{L^p(\GG_\mp; |\weight|)}^p\\
\leq& \left( p + \| \adv \|_{W^{1, \infty}(\GO)} \right) \| u \|_{W^p_\adv(\GO)}^p + \| u \|_{L^p(\GG_\mp; |\weight|)}^p,
\end{align*}
or
\[
\| u \|_{L^p(\GG_\pm; |\weight|)} \leq C_{2, p} \| u \|_{W^p_\adv(\GO)} + \| u \|_{L^p(\GG_\mp; |\weight|)}.
\]
Therefore, Lemma \ref{lem:trace_adv} is proved. 
\end{proof}

As a corollary of Lemma \ref{lem:trace_adv}, we obtain the following estimate.

\begin{prop} \label{prop:trace_adv_loc} 
Let $1 \leq p < \infty$. Also, let $\GG_\pm'$ be closed subsets of $\GG_\pm$ such that $\dist(\GG_\pm, \GG_\mp') > 0$ respectively. Then, there exists a constant $C$ depending on $\GG_\pm'$, $\adv$ and $p$ such that 
\[
\| u \|_{L^p(\GG_\pm'; |\weight|)} \leq C \| u \|_{W^p_\adv(\GO)}
\]
for all $u \in \Lip(\ol{\GO})$.
\end{prop}

\begin{proof}
We only discuss the local trace on $\GG_-$ because we can discuss that on $\GG_+$ in the same way. 

Let $\GG_-'$ be a closed subset of $\GG_-$ such that $\dist (\GG_+, \GG_-') > 0$, and let $\Gy^-$ be a function in $\Lip(\ol{\GO})$ such that $0 \leq \Gy^- \leq 1$ on $\ol{\GO}$, $\Gy^- = 0$ on $\GG_+$, and $\Gy^- = 1$ on $\GG_-'$. Also, let $u \in \Lip(\ol{\GO})$. We apply Lemma \ref{lem:trace_adv} to $\Gy^- u \in \Lip(\ol{\GO})$ to obtain
\[
\| \Gy^- u \|_{L^p(\GG_-; |\weight|)} \leq C_{2, p} \| \Gy^- u \|_{W^p_\adv(\GO)},
\]
where $C_{2, p}$ is the constant defined in \eqref{def:C2p}. By the definition of $\Gy^-$, we have $\| u \|_{L^p(\GG_-'; |\weight|)} \leq \| \Gy^- u \|_{L^p(\GG_-; |\weight|)}$ and $\| \Gy^- u \|_{L^p(\GO)} \leq \| u \|_{L^p(\GO)}$. Also, we have
\[
\| \adv \cdot \nabla (\Gy^- u) \|_{L^p(\GO)} \leq \| \adv \cdot \nabla \Gy^- \|_{L^\infty(\GO)} \| u \|_{L^p(\GO)} + \| \adv \cdot \nabla u \|_{L^p(\GO)}.
\]
Hence, we obtain
\begin{align*}
\| u \|_{L^p(\GG_-'; |\weight|)} \leq C_{2, p} (1 + \| \adv \cdot \nabla \Gy^- \|_{L^\infty(\GO)}) \| u \|_{W^p_\adv(\GO)}
\end{align*}
for all $u \in \Lip(\ol{\GO})$. This completes the proof.
\end{proof}

Since choices of $\GG_\pm'$ are arbitrary in Proposition \ref{prop:trace_adv_loc} and $\Lip(\ol{\GO})$ is dense in $W^p_\adv(\GO)$, we may consider the local trace operators $\tilde{\Gg}_\pm: W^p_\adv(\GO) \to L^p_{\loc}(\GG_\pm; |\weight|)$. It is worth mentioning that Proposition \ref{prop:trace_adv} does not imply that the traces $\tilde{\Gg}_\pm: W^p_\adv(\GO) \to L^p(\GG_\pm; |\weight|)$ with $\Dom(\tilde{\Gg}_\pm) = \Lip(\ol{\GO})$ have bounded extensions. Indeed, the following example shows that it is not possible.

\begin{example} \label{ex:unbounded}
Let $\GO := \{ (x_1, x_2) \in \Rbb^2 \mid |x_1| < x_2, 0 < x_2 < 1 \}$ and $\adv(x) := (1, 0)$. Also, let 
\[
u_m(x) := 
\begin{cases}
m^\alpha (1 - m x_2)^2, &0 \leq x_2 < 1/m,\\
0, &\mbox{otherwise},
\end{cases}
\]
where $\alpha$ is a positive number. Then, each $u_m$ belongs to $\Lip(\ol{\GO})$ and
\[
\| u_m \|_{W^p_\adv(\GO)}^p = \| u_m \|_{L^p(\GO)}^p = \frac{2}{(2p+1)(2p+2)} m^{p \alpha - 2}.
\]
Also, since $d\Gs_x = \sqrt{2} dx_2$ and $\weight = 1/\sqrt{2}$ on $\GG_+$, we have
\[
\| u_m \|_{L^p(\GG_+; \weight)}^p = \frac{1}{2p+1} m^{p\alpha - 1}. 
\] 
For fixed $p$ with $1 \leq p < \infty$, we take $\alpha$ so that $1/p < \alpha < 2/p$. Then, we have $\| u_m \|_{W^p_\adv(\GO)} \to 0$ while $\| u_m \|_{L^p(\GG_+; \weight)} \to \infty$ as $m \to \infty$. We notice that $\| u_m \|_{L^p(\GG_+'; \weight)} \to 0$ as $m \to \infty$ for any closed subset $\GG_+'$ in $\GG_+$. We remark that this example is essentially the same as the one in \cite[Remark 2.6]{DE}.

\begin{figure}[h]

\centering

\begin{tikzpicture}[scale=1.5]
	\draw[thick, ->] (-3, 0) --(3, 0);
	\draw[thick, ->] (0, -1) --(0, 3);
	\draw[thick, ->] (-0.5, 1) --(0.5, 1);
    \draw[ultra thick] (0,0) --(2,2) --(-2,2) --(0, 0);
    \draw[dashed] (2,0) --(2, 2);
    \draw[dashed] (-2, 0) --(-2, 2); 
    \node at (-0.2,1.8) {1};
    \node at (2,-0.2) {1};
    \node at (-2, -0.2) {-1};
    \node at (-0.2,-0.2) {O};
    \node [below] at (3,0) {$x$};
    \node [left] at (0,3) {$y$};
    \node at (1, 0.5) {$\GG_+$};
    \node at (-1, 0.5) {$\GG_-$};
    \node at (-0.3, 1.2) {$\adv$};
\end{tikzpicture}

\caption{The domain $\GO$ in Example \ref{ex:unbounded}}

\end{figure}

\end{example}

In order to discuss the boundary condition, we restrict the definition domains of the operators $\tilde{\Gg}_\pm$. We define trace operators $\Gg_\pm: W^p_\adv(\GO) \to L^p(\GG_\pm; |\weight|)$ with $\Dom(\Gg_\pm) = W^p_{\adv, \tr}(\GO)$ by
\[
\Gg_\pm u := \tilde{\Gg}_\pm u, \quad u \in W^p_{\adv, \tr}(\GO).
\]
We introduce the following property for $\Gg_\pm$.
 
\begin{lemma} \label{lem:trace_closed} 
For $1 \leq p < \infty$, the operators $\Gg_\pm: W^p_\adv(\GO) \to L^p(\GG_\pm; |\weight|)$ with $\Dom(\Gg_\pm) = W^p_{\adv, \tr}(\GO)$ are closed.
\end{lemma}

\begin{proof}
Let $\{ u_m \}_{m \in \Nbb}$ be a sequence in $W^p_{\adv, \tr}(\GO)$ such that $u_m \to u$ in $W^p_\adv(\GO)$ and we assume that sequences $\{ \Gg_\pm u_m \}_{m \in \Nbb}$ converge to functions $g_\pm$ in $L^p(\GG_\pm; |\weight|)$ respectively. We show that $u \in W^p_{\adv, \tr}(\GO)$ and $g_\pm = \tilde{\Gg}_\pm u$.

Let $\GG_\pm'$ be closed subsets of $\GG_\pm$ such that $\dist(\GG_\pm, \GG_\mp') > 0$ respectively. Then, by Proposition \ref{prop:trace_adv_loc}, we have 
\begin{align*}
\| g_\pm - \tilde{\Gg}_\pm u \|_{L^p(\GG_\pm'; |\weight|)} \leq& \| g_\pm - \tilde{\Gg}_\pm u_m \|_{L^p(\GG_\pm'; |\weight|)} + \| \tilde{\Gg}_\pm u_m - \tilde{\Gg}_\pm u \|_{L^p(\GG_\pm'; |\weight|)}\\
\leq& \| g_\pm - \Gg_\pm u_m \|_{L^p(\GG_\pm; |\weight|)} + C \| u_m - u \|_{W^p_\adv(\GO)} 
\end{align*}
with some positive constant $C$. Taking the limit $m \to \infty$, we obtain $\| g - \tilde{\Gg}_\pm u \|_{L^p(\GG_\pm'; |\weight|)} = 0$. Since $\GG_\pm'$ are arbitrary, we conclude that $\| g_\pm - \tilde{\Gg}_\pm u \|_{L^p(\GG_\pm; |\weight|)} = 0$, which implies that $g_\pm = \tilde{\Gg}_\pm u$. This completes the proof. 
\end{proof}

The estimate \eqref{est:Green} in Lemma \ref{lem:trace_adv} holds for all $u \in W^p_{\adv, \tr}(\GO)$ with $1 \leq p < \infty$ under the assumption \eqref{eq:density}. Moreover, we obtain the following estimate. 

\begin{prop} \label{prop:trace_adv} 
Let $1 \leq p < \infty$. Then, for a function $u \in W^p_{\adv, \pm}(\GO)$, we have
\begin{equation} \label{est:trace_Lp}
\| u \|_{L^p(\GG_\mp; |\weight|)} \leq C_{2, p} \| u \|_{W^p_\adv(\GO)},
\end{equation}
where $C_{2, p}$ is the constant defined by \eqref{def:C2p}.
\end{prop}

Recalling the definition of the norm $\| \cdot \|_{W^p_{\adv, \tr}(\GO)}$, we obtain the following estimate, which shows equivalence of two norms $\| \cdot \|_{W^p_{\adv, \tr}(\GO)}$ and $\| \cdot \|_{W^p_\adv(\GO)}$ on $W^p_{\adv, \pm}(\GO)$.

\begin{cor} \label{cor:trace_adv} 
Let $1 \leq p < \infty$. Then, for a function $u \in W^p_{\adv, \pm}(\GO)$, we have
\[
\| u \|_{W^p_{\adv, \tr}(\GO)} \leq (1 + C_{2, p}) \| u \|_{W^p_\adv(\GO)},
\]
where $C_{2, p}$ is the constant defined by \eqref{def:C2p}. 
\end{cor}

Let us mention surjectivity of the traces. Let
\[
\Lip_c(\GG_\pm) := \{ g \mid \mbox{Lipschitz continuous on } \GG_\pm, \, \supp g \subset \GG_\pm \}
\]
and
\[
\Lip_{\GG_\pm}(\ol{\GO}) := \{ u \in \Lip(\ol{\GO}) \mid u|_{\GG_\pm} = 0 \}.
\]
Then, the following surjectivity is known in \cite{C}. 

\begin{prop} \label{prop:trace_surj} 
For any $g_\pm \in \Lip_c(\GG_\pm)$, there exist functions $u_\mp \in \Lip_{\GG_\mp}(\ol{\GO})$ such that $\Gg_\pm u_\mp = g_\pm$.
\end{prop} 

We note that the proof in \cite{C} works without the separation condition because we can take an open covering of $\supp g_\pm$ such that it does not intersect with $\GG_\mp$.

\begin{remark} \label{rem:inclusion_density}
By identifying $\Lip(\ol{\GO})$ and $W^{1, \infty}(\GO)$, we see that $\Lip_{\GG_\pm}(\ol{\GO}) \subset W^p_{\adv, \pm}(\GO)$ for all $1 \leq p < \infty$, and that $\Lip_c(\GG_\pm)$ are dense in $L^p(\GG_\pm; |\weight|)$ for all $1 \leq p < \infty$ respectively.
\end{remark}

At the end of this section, we introduce the Banach-Ne\v{c}as-Babu\v{s}ka theorem, which will be used in the following sections in order to prove the well-posedness. For the detail, see \cite{2017S}.

\begin{prop}[the Banach-Ne\v{c}as-Babu\v{s}ka theorem] \label{prop:BNB} 
Let $V$ and $W$ be Banach spaces over $\Rbb$ and suppose that $W$ is reflexive. Also, let $a$ be a bounded bilinear form on $V \times W$. Then, the following three conditions are equivalent to each other:
\begin{enumerate}
\item For any $F \in W'$, there exists a unique $u \in V$ such that $a(u, w) = F(w)$ for all $w \in W$.

\item The following two conditions hold:
\begin{align*}
&c_1 := \inf_{v \in V \sm \{ 0 \}} \sup_{w \in W \sm \{ 0 \}} \frac{a(v, w)}{\| v \|_V \| w \|_W} > 0,\\
&\{ w \in W \mid a(v, w) = 0 \mbox{ for all } v \in V \} = \{ 0 \}.
\end{align*}

\item The following two conditions hold:
\begin{align*}
&c_1 := \inf_{v \in V \sm \{ 0 \}} \sup_{w \in W \sm \{ 0 \}} \frac{a(v, w)}{\| v \|_V \| w \|_W} > 0,\\
&c_2 := \inf_{w \in W \sm \{ 0 \}} \sup_{v \in V \sm \{ 0 \}} \frac{a(v, w)}{\| v \|_V \| w \|_W} > 0.
\end{align*}
\end{enumerate}
Moreover, the estimate $\| u \|_V \leq c_1^{-1} \| F \|_{W'}$ holds for $u$ in the condition 1.
\end{prop}

\begin{remark} \label{rem:equal}
In fact, if the condition 3 holds, then $c_1 = c_2$ \cite[Lemma C.53]{2021EG2}. 
\end{remark}

\section{$L^p$ well-posedness for $1 < p \leq \infty$} \label{sec:WP_Lp}

In this section, we give a proof of Theorem \ref{thm:WP_Lp}. Namely, we show well-posedness of the boundary value problem \eqref{BVP} in $L^p(\GO)$ for $1 < p \leq \infty$.

We first consider the case $1 < p < \infty$. Let $q$ be the H\"older conjugate of $p$. We consider the adjoint boundary value problem:
\beq \label{BVP_ad0}
\begin{cases}
- \nabla \cdot (\adv u) + \reac u = f &\mbox{ in } \GO,\\
u = 0 &\mbox{ on } \GG_+,
\end{cases}
\eeq
where $f \in L^q(\GO)$. If the surface measure of $\GG_+$ is $0$, then we ignore the boundary condition in \eqref{BVP_ad0}. 

\begin{lemma} \label{lem:WP_ad}
Let $1 < p < \infty$ and assume \eqref{reacp}. Then, the boundary value problem \eqref{BVP_ad0} has a unique strong solution in $W^q_{\adv, +}(\GO)$. Moreover, we have
$$
\| u \|_{W^q_\adv(\GO)} \leq C_{1, p} \| f \|_{L^q(\GO)}
$$
for all $f \in L^q(\GO)$, where $u$ is the solution to \eqref{BVP_ad0} and $C_{1, p}$ is the constant defined by \eqref{def:C1p}.
\end{lemma}

\begin{proof}
We define the operator $L': W^q_{\adv, +}(\GO) \to L^q(\GO)$ by
\[
L' u := -\nabla \cdot (\adv u) + \reac u.
\]
It is worth mentioning that
\begin{align*}
\| L' u \|_{L^q(\GO)} \leq& \| (\div \adv) u \|_{L^q(\GO)} + \| \adv \cdot \nabla u \|_{L^q(\GO)} + \| \reac u \|_{L^q(\GO)}\\ 
\leq& \left( 1 + \| \adv \|_{W^{1, \infty}(\GO)} + \| \reac \|_{L^\infty(\GO)} \right) \| u \|_{W^q_\adv(\GO)},
\end{align*}
or the operator $L'$ is bounded.

We define a bilinear form $B'$ on $W^q_{\adv, +}(\GO) \times L^p(\GO)$ by
\[
B'(u, v) := \int_\GO (L' u) v\,dx, \quad u \in W^q_{\adv, +}(\GO), v \in L^p(\GO).
\] 
By the H\"older inequality, we have
\begin{align*}
|B'(u, v)| \leq \| L' u \|_{L^q(\GO)} \| v \|_{L^p(\GO)} \leq \left( 1 + \| \adv \|_{W^{1, \infty}(\GO)} + \| \reac \|_{L^\infty(\GO)} \right) \| u \|_{W^q_\adv(\GO)} \| v \|_{L^p(\GO)},
\end{align*}
and thus the bilinear form $B'$ is also bounded. 

We also define a functional $F_0: L^p(\GO) \to \Rbb$ by
\[
F_0(v) := \la f, v \ra_{L^q(\GO), L^p(\GO)} = \int_\GO f v\,dx, \quad v \in L^p(\GO).
\]
By the H\"older inequality, it is bounded. Thus, the functional $F_0$ belongs to $L^p(\GO)'$.

We apply Proposition \ref{prop:BNB} with $V = W^q_{\adv, +}(\GO)$ and $W = L^p(\GO)$ in order to show the existence of a function $u \in W^q_{\adv, +}(\GO)$ such that
\[
B'(u, v) = F_0(v)
\]
for all $v \in L^p(\GO)$. This function $u$ is the solution to the boundary value problem \eqref{BVP_ad0}.

To this aim, we first show that 
\beq \label{inf-sup}
\inf_{u \in W^q_{\adv, +}(\GO) \setminus \{0\}} \sup_{v \in L^p(\GO) \setminus \{0\}} \frac{B'(u, v)}{\| u \|_{W^q_\adv(\GO)} \| v \|_{L^p(\GO)}} \geq C_{1, p}^{-1}.
\eeq
For $u \in W^q_{\adv, +}(\GO)$, let $v := |u|^{q-2} u$. We notice that 
\[
\int_\GO |v|^p\,dx = \int_\GO |u|^q\,dx,
\]
or $\| v \|_{L^p(\GO)} = \| u \|_{L^q(\GO)}^{q - 1}$ since $p = q/(q - 1)$. Moreover, Green's formula yields
\begin{align*}
B'(u, v) =& \int_\GO (- \nabla \cdot (\adv u) + \reac u) |u|^{q-2} u\,dx\\
=& \int_\GO \left( \reac - \frac{1}{p} \div \adv \right) |u|^q\,dx + \frac{1}{q} \int_{\GG_-} |u|^q |\weight|\,d\Gs_x\\ 
\geq& \reac_p \| u \|_{L^q(\GO)}^q.
\end{align*}
Thus, we have
\[
\| u \|_{L^q(\GO)}^q \leq \reac_p^{-1} B'(u, v) \leq \reac_p^{-1} \| u \|_{L^q(\GO)}^{q - 1} \sup_{v \in L^p(\GO) \setminus \{ 0 \}} \frac{B'(u, v)}{\| v \|_{L^p(\GO)}},
\]
or
\[
\| u \|_{L^q(\GO)} \leq \reac_p^{-1} \sup_{v \in L^p(\GO) \setminus \{ 0 \}} \frac{B'(u, v)}{\| v \|_{L^p(\GO)}}.
\]
For the derivative term, we have
\begin{align*}
\| \adv \cdot \nabla u \|_{L^q(\GO)} =& \sup_{v \in L^p(\GO) \setminus \{ 0 \}} \frac{\la \adv \cdot \nabla u, v \ra_{L^q(\GO), L^p(\GO)}}{\| v \|_{L^p(\GO)}}\\
=& \sup_{v \in L^p(\GO) \setminus \{ 0 \}} \frac{-B'(u, v) - \la (\div \adv)u, v \ra_{L^q(\GO), L^p(\GO)} + \la \reac u, v \ra_{L^q(\GO), L^p(\GO)}}{\| v \|_{L^p(\GO)}}\\
\leq& \left( 1 + \reac_p^{-1} (\| \adv \|_{W^{1, \infty}(\GO)} + \| \reac \|_{L^\infty(\GO)} ) \right) \sup_{v \in L^p(\GO) \setminus \{ 0 \}} \frac{B'(u, v)}{\| v \|_{L^p(\GO)}}.
\end{align*}
Therefore, by Corollary \ref{cor:trace_adv}, we obtain
\[
\| u \|_{W^q_{\adv, \tr}(\GO)} \leq (1 + C_{2, p}) \| u \|_{W^q_\adv(\GO)} \leq C_{1, p} \sup_{v \in L^p(\GO) \setminus \{ 0 \}} \frac{B'(u, v)}{\| v \|_{L^p(\GO)}},
\]
which implies the estimate \eqref{inf-sup}.

We next show that that $B'(u, v) = 0$ for all $u \in W^q_{\adv, +}(\GO)$ implies that $v = 0$. Since $W^{1, q}_0(\GO) \subset W^q_{\adv, +}(\GO)$, this assumption implies that the function $v$ satisfies $\adv \cdot \nabla v + \reac v = 0$ in $W^{-1, p}(\GO)$. Thus, we have $\adv \cdot \nabla v = - \reac v$ and the directional derivative $\adv \cdot \nabla v$ belongs to $L^p(\GO)$. In other words, we have $v \in W^p_\adv(\GO)$. This fact allows us to take the trace $\tilde{\Gg}_- v$ in $L^p_{\loc}(\GG_-; |\weight|)$. 

We take a closed subset $\GG_-' \subset \GG_-$ such that $\dist(\GG_+, \GG_-') > 0$ and take $u \in W^q_{\adv, +}(\GO)$ such that $\supp (\Gg_- u) \subset \GG_-'$. Then, we may do integration by parts to obtain
\[
\la L' u, v \ra_{L^q(\GO), L^p(\GO)} = \la u, Lv \ra_{L^q(\GO), L^p(\GO)} + \int_{\GG_-'} u v |\weight|\,d\Gs_x,
\]
where $Lv := \adv \cdot \nabla v + \reac v$. Since $Lv = 0$, the above formula implies
\[
\int_{\GG_-'} u v |\weight|\,d\Gs_x = 0,
\]
and, by Proposition \ref{prop:trace_surj} with Remark \ref{rem:inclusion_density}, we see that $\tilde{\Gg}_- v = 0$, which implies that $v \in W^p_{\adv, -}(\GO)$. 

Therefore, we have
\begin{align*}
0 =& \int_\GO (\adv \cdot \nabla v + \reac v) |v|^{p - 2} v\,dx\\
=& \int_\GO \left( \reac - \frac{1}{p} \div \adv \right) |v|^p\,dx + \frac{1}{p} \int_{\GG_+} |v|^q \weight\,d\Gs_x\\
\geq& \reac_p \| v \|_{L^p(\GO)}^p,
\end{align*}
which means that $v = 0$.

The conclusion follows from Proposition \ref{prop:BNB} and the fact that $\| F_0 \|_{L^p(\GO)'} = \| f \|_{L^q(\GO)}$.
\end{proof} 

We are ready to prove Theorem \ref{thm:WP_Lp} for $1 < p < \infty$. Corresponding to the equation \eqref{BVP_weak}, we let
\[
F(v) := \la f, v \ra_{W^q_{\adv, +}(\GO)', W^q_{\adv, +}(\GO)} + \int_{\GG_-} g v |\weight|\,d\Gs_x, \quad v \in W^q_{\adv, +}(\GO).
\]
Since
\[
\left| \la f, v \ra_{W^q_{\adv, +}(\GO)', W^q_{\adv, +}(\GO)} \right| \leq \| f \|_{W^q_{\adv, +}(\GO)'} \| v \|_{W^q_{\adv, \tr}(\GO)} 
\]
and 
\[
\left| \int_{\GG_-} g v |\weight| \,d\Gs_x \right| \leq \| g \|_{L^p(\GG_-; |\weight|)} \| v \|_{L^q(\GG_-; |\weight|)} \leq \| g \|_{L^p(\GG_-; |\weight|)} \| v \|_{W^q_{\adv, \tr}(\GO)},
\]
we have $F \in W^q_{\adv, +}(\GO)'$ and
\[
\| F \|_{W^p_{\adv, +}(\GO)'} \leq \| f \|_{W^q_{\adv, +}(\GO)'} + \| g \|_{L^p(\GG_-; |\weight|)}.
\]

We define the operator $\ol{L}: L^p(\GO) \to W^q_{\adv, +}(\GO)'$ by
\[
\la \ol{L} u, v \ra_{W^q_{\adv, +}(\GO)', W^q_{\adv, +}(\GO)} := \la u, L' v \ra_{L^p(\GO), L^q(\GO)},
\]
and define the bilinear form $\ol{B}$ on $L^p(\GO) \times W^q_{\adv, +}(\GO)$ by
\[
\ol{B}(u, v) := \la \ol{L} u, v \ra_{W^q_{\adv, +}(\GO)', W^q_{\adv, +}(\GO)}.
\]
It is worth mentioning that $\ol{B}(u, v) = B'(v, u)$ for all $u \in L^p(\GO)$ and $v \in W^q_{\adv, +}(\GO)$. The equation \eqref{BVP_weak} is rewritten as
\[
\ol{B}(u, v) = F(v)
\]
for all $v \in W^q_{\adv, +}(\GO)$.

Lemma \ref{lem:WP_ad} and Proposition \ref{prop:BNB} with Remark \ref{rem:equal} imply that 
\[
\inf_{v \in W^q_{\adv, +}(\GO) \setminus \{ 0 \}} \sup_{u \in L^p(\GO) \setminus \{ 0 \}} \frac{\ol{B}(u, v)}{\| u \|_{L^p(\GO)} \| v \|_{W^q_\adv(\GO)}} \geq C_{1, p}^{-1}
\]
and
\[
\inf_{u \in L^p(\GO) \setminus \{ 0 \}} \sup_{v \in W^q_{\adv, +}(\GO) \setminus \{ 0 \}} \frac{\ol{B}(u, v)}{\| u \|_{L^p(\GO)} \| v \|_{W^q_\adv(\GO)}} \geq C_{1, p}^{-1}.
\]
We apply Proposition \ref{prop:BNB} again in order to conclude that
\[
\| u \|_{L^p(\GO)} \leq C_{1, p} \| F \|_{W^p_{\adv, +}(\GO)'} \leq C_{1, p} \left( \| f \|_{W^q_{\adv, +}(\GO)'} + \| g \|_{L^p(\GG_-; |\weight|)} \right).
\]
This completes the proof for $1 < p < \infty$.

We proceed  to the case $p = \infty$. We note that the same approach is seen in \cite{SP}. Let $p_\infty := (2 \| \adv \|_{W^{1, \infty}(\GO)}) / \reac_\infty$. Then, we get
\[
\reac_p \geq \reac_\infty - \frac{1}{p_\infty} \| \adv \|_{W^{1, \infty}(\GO)} = \frac{\reac_\infty}{2} > 0
\]
for all $p > p_\infty$. Also, since
\[
\reac_\infty - \frac{1}{p} \| \adv \|_{W^{1, \infty}(\GO)} \leq \reac_p \leq \reac_\infty + \frac{1}{p} \| \adv \|_{W^{1, \infty}(\GO)},
\]
we have
\begin{align*}
C_{1, p} \leq& \left( 1 + p^{\frac{1}{p}} + \| \adv \|_{W^{1, \infty}(\GO)}^{\frac{1}{p}} \right) \left( \reac_\infty - \frac{1}{p} \| \adv \|_{W^{1, \infty}(\GO)} \right)^{-1}\\
&\times \left( 1 + \reac_\infty + \frac{1}{p} \| \adv \|_{W^{1, \infty}(\GO)} + \| \reac \|_{L^\infty(\GO)} \right).
\end{align*}
for all $p > p_\infty$. Thus, we obtain
\[
\limsup_{p \to \infty} C_{1, p} \leq C_{1, \infty},
\]
where $C_{1, \infty}$ is the constant defined in \eqref{def:C1_infty}.

Let $f \in W^1_{\adv, +}(\GO)'$. Since $W^q_{\adv, +}(\GO) \subset W^1_{\adv, +}(\GO)$ for all $q > 1$, we can define $f_q \in W^q_{\adv, +}(\GO)'$ by
\[
\la f_q, v \ra_{W^q_{\adv, +}(\GO)', W^q_{\adv, +}(\GO)} := \la f, v \ra_{W^1_{\adv, +}(\GO)', W^1_{\adv, +}(\GO)}, \quad v \in W^q_{\adv, +}(\GO).
\]
Then, we have
\begin{align*}
|\la f_q, v \ra_{W^q_{\adv, +}(\GO)', W^q_{\adv, +}(\GO)}| \leq& \| f \|_{W^1_{\adv, +}(\GO)'} \| v \|_{W^1_{\adv, \tr}(\GO)}\\ 
\leq& \left( |\GO|^{\frac{1}{p}} + \| \adv \|_{W^{1, \infty}(\GO)}^{\frac{1}{p}} |\GG_-|^{\frac{1}{p}} \right) \| f \|_{W^1_{\adv, +}(\GO)'} \| v \|_{W^q_{\adv, +}(\GO)},
\end{align*}
and hence, we obtain
\[
\| f_q \|_{W^1_{\adv, +}(\GO)'} \leq \left( |\GO|^{\frac{1}{p}} + \| \adv \|_{W^{1, \infty}(\GO)}^{\frac{1}{p}} |\GG_-|^{\frac{1}{p}} \right) \| f \|_{W^1_{\adv, +}(\GO)'}. 
\]

Since $g \in L^\infty(\GG_-; |\weight|)$, we have
\[
\| g \|_{L^p(\GG_-; |\weight|)} \leq \| \adv \|_{W^{1, \infty}(\GO)}^{\frac{1}{p}} |\GG_-|^{\frac{1}{p}} \| g \|_{L^\infty(\GG_-; |\weight|)}.
\]
Regarding $f = f_q$ for $q > 1$ and applying Theorem \ref{thm:WP_Lp} for $p_\infty < p < \infty$, we obtain
\begin{align*}
\| u \|_{L^p(\GO)} \leq& C_{1, p} \left( |\GO|^{\frac{1}{p}} \| f \|_{W^1_{\adv, +}(\GO)'} + \| \adv \|_{W^{1, \infty}(\GO)}^{\frac{1}{p}} |\GG_-|^{\frac{1}{p}} \| g \|_{L^\infty(\GG_-)} \right)
\end{align*}
for all $p > p_\infty$. Taking $p \to \infty$, we have
\[
\| u \|_{L^\infty(\GO)} \leq C_{1, \infty} \left( \| f \|_{W^1_{\adv, +}(\GO)'} + \| g \|_{L^\infty(\GG_-)} \right).
\]
This completes the proof for $p = \infty$.

Therefore, Theorem \ref{thm:WP_Lp} is proved.

\section{$W^{1, p}$ well-posedness for $1 \leq p \leq \infty$} \label{sec:WP_Wp}

In this section, we prove Theorem \ref{thm:WP_Wp}. In other words, we discuss well-posedness of the boundary value problem \eqref{BVP} in $W^p_\adv(\GO)$ for $1 \leq p \leq \infty$.

We first discuss the case $1 < p < \infty$.

\begin{lemma} \label{lem:Lp_to_Wp}
Let $1 < p < \infty$ and assume \eqref{reacp}. Then, for any $f \in L^p(\GO)$ and $g \in L^p(\GG_-; |\weight|)$, the weak solution $u \in L^p(\GO)$ to the boundary value problem \eqref{BVP} belongs to $W^p_\adv(\GO)$. Moreover, we have $u|_{\GG_-} = g$ and 
\[
\| u \|_{W^p_\adv(\GO)} \leq \{ 1 + \left( 1 + \| \reac \|_{L^\infty(\GO)} \right) C_{1, p} \} \left( \| f \|_{L^p(\GO)} + \| g \|_{L^p(\GG_-; |\weight|)} \right),
\]
where $C_{1, p}$ is the constant defined by \eqref{def:C1p}.
\end{lemma}

\begin{proof}
Let $u \in L^p(\GO)$ be the solution to the equation \eqref{BVP_weak}. As in the proof of Lemma \ref{lem:WP_ad}, we can say that 
\beq \label{eq:SAE}
\adv \cdot \nabla u = f - \reac u \in L^p(\GO).
\eeq
Thus, the solution $u$ belongs to $W^p_\adv(\GO)$. Moreover, we have 
\begin{align*}
\| \adv \cdot \nabla u \|_{L^p(\GO)} \leq \| f \|_{L^p(\GO)} + \| \reac u \|_{L^p(\GO)} \leq \| f \|_{L^p(\GO)} + \| \reac \|_{L^\infty(\GO)} \| u \|_{L^p(\GO)}.
\end{align*}
Thus, by Theorem \ref{thm:WP_Lp}, we obtain
\begin{align*}
\| u \|_{W^p_\adv(\GO)} \leq& \| f \|_{L^p(\GO)} + (1 + \| \reac \|_{L^\infty(\GO)}) \| u \|_{L^p(\GO)}\\ 
\leq& \{ 1 + \left( 1 + \| \reac \|_{L^\infty(\GO)} \right) C_{1, p} \} \left( \| f \|_{L^p(\GO)} + \| g \|_{L^p(\GG_-; |\weight|)} \right).
\end{align*}

We check that the function $u$ has the trace on $\GG_-$ and it is equal to the boundary data $g$. Let $\GG_-'$ be a closed subset of $\GG_-$ such that $\dist(\GG_+, \GG_-') > 0$, and let $v$ be a function in $\Lip_{\GG_+}(\GO)$ such that $\supp v|_{\GG_-} \subset \GG_-'$. Then, integration by parts yields
\[
\int_\GO u \left( -\nabla \cdot (\adv v) \right)\,dx = \int_{\GG_-'} u v |\weight|\,d\Gs_x + \int_\GO (\adv \cdot u) v\,dx.
\]
Applying the above identity to the equation \eqref{BVP_weak}, thanks to \eqref{eq:SAE}, we have
\[
\int_{\GG_-'} uv |\weight|\,d\Gs_x = \int_{\GG_-'} g v |\weight|\,d\Gs_x
\]
for all $v \in \Lip_{\GG_+}(\GO)$ such that $\supp v|_{\GG_-} \subset \GG_-'$. By Proposition \ref{prop:trace_surj}, we conclude that
\[
\tilde{\Gg}_- u = g \in L^p(\GG_-; |\weight|). 
\]

This completes the proof.
\end{proof}

We derive the estimate in Theorem \ref{thm:WP_Wp}. Let $u \in W^p_{\adv, \tr}(\GO)$ be the strong solution to the boundary value problem \eqref{BVP}. By Lemma \ref{lem:trace_adv}, we see that
\[
\| u \|_{L^p(\GG_+; \weight)} \leq C_{2, p} \| u \|_{W^p_\adv(\GO)} + \| u \|_{L^p(\GG_-; |\weight|)}.
\]
Hence, we have
\begin{align*}
\| u \|_{W^p_{\adv, \tr}(\GO)} \leq& (1 + C_{2, p}) \| u \|_{W^p_\adv(\GO)} + \| g \|_{L^p(\GG_-; |\weight|)}\\
\leq& C_{1, p}' \left( \| f \|_{L^p(\GO)} + \| g \|_{L^p(\GG_-; |\weight|)} \right),
\end{align*}
which is the estimate in Theorem \ref{thm:WP_Wp}.

We can treat the case $p = \infty$ in the same way as the proof of the $L^\infty$ well-posedness in Section \ref{sec:WP_Lp}. In what follows, we discuss the case $p = 1$.

We first show existence of a solution in $W^1_\adv(\GO)$. For $f \in L^1(\GO)$ and $N \in \Nbb$, let $f_N := \max \{ -N, \min \{ N, f \} \}$, namely,
\begin{equation*} 
f_N(x) =
\begin{cases}
f(x), &|f(x)| \leq N,\\
N, &f(x) > N,\\
-N, &f(x) < -N.
\end{cases}
\end{equation*}
In the same way, for $g \in L^1(\GG_-; |\weight|)$, let $g_N := \max \{ -N, \min \{ N, g \} \}$. It is worth mentioning that $f_N \in L^\infty(\GO)$ and $g_N \in L^\infty(\GG_-)$ for each $N \in \Nbb$, and that
\[
\| f - f_N \|_{L^1(\GO)} \to 0, \quad \| g - g_N \|_{L^1(\GG_-; |\weight|)} \to 0
\]
as $N \to \infty$ by the dominated convergence theorem.

Assume \eqref{reacp} with $p = 1$ and let
\[
p_1 := 1 + \frac{\reac_1}{2 \| \adv \|_{W^{1, \infty}(\GO)}}.
\]
Then, we have $p_1 > 1$ and
\[
\reac_p \geq \reac_1 - (p_1 - 1) \| \adv \|_{W^{1, \infty}(\GO)} = \frac{\reac_1}{2} > 0
\]
for all $1 < p < p_1$. Since
\[
\reac_1 - \frac{p - 1}{p} \| \adv \|_{W^{1, \infty}(\GO)} \leq \reac_p \leq \reac_1 + \frac{p - 1}{p} \| \adv \|_{W^{1, \infty}(\GO)},
\]
we have
\begin{align*}
C_{1, p} \leq& \left( 1 + p^{\frac{1}{p}} + \| \adv \|_{W^{1, \infty}(\GO)}^{\frac{1}{p}} \right) \left( \reac_1 - \frac{p-1}{p} \| \adv \|_{W^{1, \infty}(\GO)} \right)^{-1}\\
&\times \left( 1 + \reac_1 + \frac{p - 1}{p} \| \adv \|_{W^{1, \infty(\GO)}} + \| \reac \|_{L^\infty(\GO)} \right) 
\end{align*}
for all $1 < p < p_1$. Thus, we obtain
\[
\limsup_{p \to 1} C_{1, p} \leq C_{1, 1},
\]
where
\[
C_{1, 1} := \left( 2 + \| \adv \|_{W^{1, \infty}(\GO)} \right) \reac_1^{-1} \left( 1 + \reac_1 + \| \adv \|_{W^{1, \infty(\GO)}} + \| \reac \|_{L^\infty(\GO)} \right).
\]

Since
\[
\| f_N \|_{L^p(\GO)} \leq N^{\frac{p - 1}{p}} \| f_N \|_{L^1(\GO)}^\frac{1}{p}, \quad \| g_N \|_{L^p(\GG_-; |\weight|)} \leq N^{\frac{p - 1}{p}} \| g_N \|_{L^1(\GG_-; |\weight|)}^\frac{1}{p},
\]
we have
\begin{align*}
\| u_N \|_{W_\adv^p(\GO)} \leq& \left( p C_{1, 1} + (p - 1) C_{1, 1}'' \right) N^{\frac{p - 1}{p}} \left( \| f_N \|_{L^1(\GO)}^{\frac{1}{p}} + \| g_N \|_{L^1(\GG_-; |\weight|)}^{\frac{1}{p}} \right)
\end{align*}
for all $1 < p < p_1$, where $u_N$ is the solution to the boundary value problem \eqref{BVP} with $f = f_N$ and $g = g_N$. Taking $p \to 1$, we get
\begin{equation} \label{est:Cauchy_1}
\| u_N \|_{W_\adv^1(\GO)} \leq C_{1, 1} (\| f_N \|_{L^1(\GO)} + \| g_N \|_{L^1(\GG_-; |\weight|)}).
\end{equation}
Since $\{ f_N \}$ and $\{ g_N \}$ are Cauchy sequences in $L^1(\GO)$ and $L^1(\GG_-; |\weight|)$ respectively, the estimate \eqref{est:Cauchy_1} implies that $\{ u_N \}$ is a Cauchy sequence in $W^1_\adv(\GO)$. Thus, there exists a function $u \in W^1_\adv(\GO)$ such that $u_N \to u$ in $W^1_\adv(\GO)$ as $N \to \infty$.

As the next step, we prove that the function $u$ is a solution to the boundary value problem \eqref{BVP}. Since $u_N$ satisfies $\adv \cdot \nabla u_N + \reac u_N = f_N$, we have
\begin{align*}
\| \adv \cdot \nabla u + \reac u - f \|_{L^1(\GO)} \leq& \| \adv \cdot \nabla (u - u_N) \|_{L^1(\GO)} + \| \reac (u - u_N) \|_{L^1(\GO)} + \| f - f_N \|_{L^1(\GO)}\\
\leq& \left( 1 + \| \reac \|_{L^\infty(\GO)} \right) \| u - u_N \|_{W^1_\adv(\GO)} + \| f - f_N \|_{L^1(\GO)}
\end{align*}
for all $N \in \Nbb$. Thus, by taking $N \to \infty$, we conclude that $\adv \cdot \nabla u + \reac u - f = 0$.

Now we take a closed subset $\GG_-' \subset \GG_-$ such that $\dist(\GG_+, \GG_-') > 0$ and take $\Gc \in \Lip(\ol{\GO})$ such that $\supp \Gc|_{\p \GO} \subset \GG_-'$. Then, we have
\begin{align*}
\left| \int_{\GG_-'} \Gc (u - g) \weight\,d\Gs_x \right| \leq& \left| \int_{\GG_-'} \Gc (u - u_N) \weight\,d\Gs_x \right| + \left| \int_{\GG_-'} \Gc (g_N - g) \weight\,d\Gs_x \right|\\
\leq& \left| \int_\GO \Gc \adv \cdot \nabla (u - u_N)\,dx \right| + \left| \int_\GO (\div (\Gc \adv)) (u - u_N)\,dx \right|\\ 
&+ \left| \int_{\GG_-'} \Gc (g_N - g) \weight\,d\Gs_x \right|\\
\leq& \left( \| \Gc \|_{L^\infty(\GO)} + \| \adv \cdot \nabla \Gc \|_{L^\infty(\GO)} + \| \Gc \|_{L^\infty(\GO)} \| \adv \|_{W^{1, \infty}(\GO)} \right)\\
&\times \| u - u_N \|_{W^1_\adv(\GO)} + \| \Gc \|_{L^\infty(\GO)} \| g - g_N \|_{L^1(\GG_-; |\weight|)}
\end{align*}
for all $N \in \Nbb$. Thus, by taking $N \to \infty$, we have
\[
\int_{\GG_-'} \Gc (u - g) \weight\,d\Gs_x = 0.
\]
Since $\Gc$ and $\GG_-'$ are arbitrary, we conclude that $u = g$ a.e. on $\GG_-$. Therefore, the function $u$ is the solution to the boundary value problem \eqref{BVP}.

Finally, we discuss uniqueness of solutions to the boundary value problem \eqref{BVP} in $W^1_\adv(\GO)$ following \cite{BCJK}. Let $u_1$ and $u_2$ be two solutions to the boundary value problem \eqref{BVP} in $W^1_\adv(\GO)$. Then, the difference $\tilde{u} := u_1 - u_2$ belongs to $W^1_{\adv, -}(\GO)$ and solves the following boundary value problem: 
\beq \label{BVP_uniq}
\begin{cases}
\adv \cdot \nabla \tilde{u} + \reac \tilde{u} = 0 &\mbox{ in } \GO,\\
\tilde{u} = 0 &\mbox{ on } \GG_-.
\end{cases}
\eeq
By the definition of $W^1_{\adv, -}(\GO)$, for any $\Ge > 0$, there exists a function $\tilde{u}^\Ge \in \Lip(\ol{\GO})$ such that $\| \tilde{u} - \tilde{u}^\Ge \|_{W^1_\adv(\GO)} < \Ge$ and $\| \tilde{u}^\Ge \|_{L^1(\GG_-; |\weight|)} < \Ge$. Also, let $f^\Ge := \adv \cdot \nabla \tilde{u}^\Ge + \reac \tilde{u}^\Ge$. Noting that $\adv \cdot \nabla \tilde{u} + \reac \tilde{u} = 0$, we have
\[
\| f^\Ge \|_{L^1(\GO)} \leq \|\adv \cdot \nabla (\tilde{u}^\Ge - \tilde{u}) \|_{L^1(\GO)} + \| \reac (\tilde{u}^\Ge - \tilde{u}) \|_{L^1(\GO)} \leq \left( 1 + \| \reac \|_{L^\infty(\GO)} \right) \| \tilde{u}^\Ge - \tilde{u} \|_{W^1_\adv(\GO)}.
\]
Thus, $f^\Ge \to 0$ in $L^1(\GO)$ as $\Ge \to 0$.

Let $\GO_\pm^\Ge := \{ x \in \GO \mid \pm \tilde{u}^\Ge(x) > 0 \}$. Since $\tilde{u}^\Ge \in \Lip(\ol{\GO})$, we have
\begin{align*}
\int_{\GO_\pm^\Ge} f^\Ge\,dx =& \int_{\GO_\pm^\Ge} (\adv \cdot \nabla \tilde{u}^\Ge + \reac \tilde{u}^\Ge)\,dx\\ 
=& \int_{\GO_\pm^\Ge} (\reac - \div \adv) \tilde{u}^\Ge\,dx + \int_{{\GG_+} \cap \ol{\GO_\pm^\Ge}} u^\Ge \weight\,d\Gs_x + \int_{{\GG_-} \cap \ol{\GO_\pm^\Ge}} u^\Ge \weight\,d\Gs_x.
\end{align*}
Therefore, we have
\begin{align*}
&\reac_1 \| \tilde{u}^\Ge \|_{L^1(\GO)} + \| \tilde{u}^\Ge \|_{L^1(\GG_+; |\weight|)}\\
\leq& \int_{\GO_\pm^\Ge} (\reac - \div \adv) |\tilde{u}^\Ge|\,dx + \int_{{\GG_+} \cap \ol{\GO_\pm^\Ge}} |u^\Ge| \weight\,d\Gs_x\\
=& \int_{\GO_+^\Ge} (\reac - \div \adv) \tilde{u}^\Ge\,dx - \int_{\GO_-^\Ge} (\reac - \div \adv) \tilde{u}^\Ge\,dx + \int_{{\GG_+} \cap \ol{\GO_+^\Ge}} \tilde{u}^\Ge \weight\,d\Gs_x - \int_{{\GG_+} \cap \ol{\GO_-^\Ge}} \tilde{u}^\Ge \weight\,d\Gs_x\\
=& \int_{\GO_+^\Ge} f^\Ge\,dx - \int_{\GO_-^\Ge} f^\Ge\,dx - \int_{{\GG_-} \cap \ol{\GO_+^\Ge}} \tilde{u}^\Ge \weight\,d\Gs_x + \int_{{\GG_-} \cap \ol{\GO_-^\Ge}} \tilde{u}^\Ge \weight\,d\Gs_x\\
\leq& \| f^\Ge \|_{L^1(\GO)} + \| \tilde{u}^\Ge \|_{L^1(\GG_-; |\weight|)},
\end{align*}
which implies that $\tilde{u}^\Ge \to 0$ in $L^1(\GO)$ and $\tilde{u}^\Ge|_{\GG_+} \to 0$ in $L^1(\GG_+; \weight)$ as $\Ge \to 0$. Since
\[
\| \tilde{u} \|_{L^1(\GO)} \leq \| \tilde{u} - \tilde{u}^\Ge \|_{L^1(\GO)} + \| \tilde{u}^\Ge \|_{L^1(\GO)},
\]
we conclude that $\tilde{u} = 0$, which shows the uniqueness of the solution to \eqref{BVP} in $W^1_\adv(\GO)$. 

The estimate for the strong solution with $p = 1$ in Theorem \ref{thm:WP_Wp} is derived in the same way as for the case $1 < p < \infty$. Therefore, Theorem \ref{thm:WP_Wp} is proved.

\section{A sufficient condition for the density assumption \eqref{eq:density}} \label{sec:sufficient}

We have established well-posedness results on $L^p(\GO)$ and $W^p_{\adv, \tr}(\GO)$ under the assumption \eqref{eq:density}. In this section, we give a sufficient condition for it to hold.

If there is no point in $\ol{\GG_+} \cap \ol{\GG_-}$, then the boundaries $\GG_\pm$ are well-separated, and the density \eqref{eq:density} holds \cite{2004EG}. In what follows, we consider the case where a point $\ol{\GG_+} \cap \ol{\GG_-}$ does exist.

We assume that the intersection $\ol{\GG_+} \cap \ol{\GG_-}$ consists of finite number of connected components $\GG_k$, $k = 1, \ldots, K$ with Lipschitz boundaries $\partial \GG_k$, and there exist open sets $N_k$, $k = 1, \ldots, K$, of $\Rbb^d$ such that $\GG_k \subset N_k$ for all $k = 1, \ldots, K$ and $N_{k_1} \cap N_{k_2} = \emptyset$ if $k_1 \neq k_2$. Also, let $\GO_k := \GO \cap N_k$ and $\GG_{\pm, k} := \GG_\pm \cap N_k$. We further assume that, for each $k$, either of the following two cases holds: (i) For any $T > 0$, there exists a closed subset $\GG_{-, k}'$ of $\p \GO$ with the Lipschitz boundary $\partial \GG_{-, k}'$ such that $\GG_k \subset \GG_{-, k}' \subset \GG_{-, k} \cup \GG_k$, $|\GG_{-, k}'| > 0$, and all of the solutions $x(t)$ to the Cauchy problem \eqref{Cauchy_boundary} with $x_0 \in \GG_{-, k}'$ reaches at $\GG_+$ before $t = T$. (ii) There exists $T > 0$ such that for any closed subset $\GG_{-, k}'$ of $\p \GO$ with the Lipschitz boundary $\partial \GG_{-, k}'$ and with $\GG_k \subset \GG_{-, k}' \subset \GG_{-, k} \cup \GG_k$, all of the solutions $x(t)$ to the Cauchy problem \eqref{Cauchy_boundary} with $x_0 \in \GG_{-, k}'$ cannot reach at $\GG_+$ before $t = T$. By renumbering, we assume that first $K_1$ components of $\GG_k$ correspond to the case (i), and the other $K - K_1$ components correspond to the case (ii).

We first investigate the case (i). Let $\GG_{-, k, 1}'$ and $\GG_{-, k, 2}'$ be closed subsets in $\GG_{-, k} \cup \GG_k$ with Lipschitz boundaries $\partial \GG_{-, k, j}'$ such that $\GG_k \subset \GG_{-, k, 1}' \subset \GG_{-, k, 2}'$, $|\GG_{-, k, 1}'| > 0$, and $d((\partial \GG_{-, k, 1}') \setminus \GG_k, (\partial \GG_{-, k, 2}') \setminus \GG_k) > 0$. Also, let $g$ be a Lipschitz function on $\GG_{-, k} \cup \GG_k$ such that, $0 \leq g \leq 1$, $g(x) = 0$ for $x \in \GG_{-, k, 1}'$ and $g(x) = 1$ for $(\GG_{-, k} \cup \GG_k) \setminus \GG_{-, k, 2}'$. We construct a solution $\Gy_k$ to the boundary value problem
\begin{equation} \label{eq:BVP_characteristic}
\begin{cases}
\adv \cdot \nabla \Gy_k = 0 &\mbox{ in } \GO_k,\\
\Gy_k = g &\mbox{ on } \GG_{-, k}
\end{cases}
\end{equation}
on $\ol{\GO_k}$ by the method of characteristic lines. By extending $\Gy_k$ upto $\ol{\GO}$ by $\Gy_k(x) = 1$ for $x \in \ol{\GO} \setminus N_k$ and $\Gy_k(x) = 0$ for $x \in \GG_k$, we regard it as a Lipschitz function on $\ol{\GO}$. We notice that $0 \leq \Gy_k(x) \leq 1$ for all $x \in \ol{\GO}$, and by the assumption on $\GG_k$, the function $\Gy_k$ vanishes near a smaller neighborhood of $\GG_k$ in $\Rbb^d$. 

Let $\Gy := \prod_{k = 1}^{K_1} \Gy_k$ when $K_1 \geq 1$ and $\Gy = 1$ when $K_1 = 0$. Then, we have $\adv \cdot \nabla \Gy(x) = 0$ for all $x \in \GO$ and $\Gy(x) \to 1$ for all $x \in \ol{\GO} \setminus (\cup_{k = 1}^{K_1} \GG_k )$ as $\max_{1 \leq k \leq K_1} |\GG_{-, k, 2}'| \to 0$. Let $u \in W^p_{\adv, \tr}(\GO)$. Then, based on the above properties of $\Gy$, we have 
\begin{align*}
\| u - \Gy u \|_{W^p_{\adv, \tr}(\GO)}^p =& \| (1 - \Gy) u \|_{L^p(\GO)}^p + \| (1 - \Gy) \adv \cdot \nabla u \|_{L^p(\GO)}^p\\ 
&+ \| (1 - \Gy) \tilde{\Gg}_+ u \|_{L^p(\GG_+; \weight)}^p + \| (1 - \Gy) \tilde{\Gg}_- u \|_{L^p(\GG_-; |\weight|)}^p.
\end{align*}
When $K_1 \geq 1$, Lebesgue's convergence theorem implies that, for any $\Ge > 0$, there exist $\GG_{-, k, 1}'$ and $\GG_{-, k, 2}'$, $k = 1, \ldots, K_1$, such that $\| u - \Gy u \|_{W^p_{\adv, \tr}(\GO)} < \Ge$. Also, $u = \Gy u$ when $K_1 = 0$. Thus, in what follows, we consider approximation of $\Gy u$ instead of $u$.

We next investigate the case (ii). This case does not occur when $K_1 = K$. Thus, in what follows, we assume that $k_1 < K$. 

Let $M_{k, 1}$ and $M_{k, 2}$ be open neighborhoods of $\GG_k$ in $\Rbb^d$ with Lipschitz boundaries such that $\GG_k \subset M_{k, 1}$, $\ol{M_{k, 1}} \subset M_{k, 2}$, and $\ol{M_{k, 2}} \subset N_k$. It is worth mentioning that $d(\p M_{k, 1}, \p M_{k, 2}) > 0$ and $d(\p M_{k, 2}, \p N_k) > 0$. 

We introduce two kinds of cut-off functions. For $k = K_1 + 1, \ldots, K$, let $\Gf_k \in \Lip(\ol{\GO})$ such that $0 \leq \Gf_k \leq 1$, $\Gf_k(x) = 1$ for $x \in \ol{M_{k, 2}}$, and $\supp \Gf_k \subset N_k$. Also, let $\Gf := 1 - \sum_{k=K_1 + 1}^K \Gf_k$. Furthermore, let $\Gc_k \in \Lip(\ol{\GO})$ such that $0 \leq \Gc_k \leq 1$, $\Gc_k(x) = 1$ for $x \in \ol{M_{k, 1}}$, and $\supp \Gc_k \subset M_{k, 2}$. 

For $u \in W^p_{\adv, \tr}(\GO)$, there exists a sequence $\{ u_m \} \subset \Lip(\ol{\GO})$ such that $\| u - u_m \|_{W^p_\adv(\GO)} \to 0$ as $m \to \infty$. For the function $u$ and the sequence $\{ u_m \}$, let $u_0 := \Gf \Gy u$, $u_k := \Gf_k \Gy u$, and $u_{m, 0} := \Gf \Gy u_m$, $u_{m, k} := \Gf_k \Gy u_m$, $k = K_1 + 1, \ldots, K$, respectively. Then, we have $u_{m, 0}, u_{m, k} \in \Lip(\ol{\GO})$ and
\[
\| \Gy(u - u_m) \|_{W^p_{\adv, \tr}(\GO)} \leq \| u_0 - u_{m, 0} \|_{W^p_{\adv, \tr}(\GO)} + \sum_{k = K_1 + 1}^K \| u_k - u_{m, k} \|_{W^p_{\adv, \tr}(\GO)}.
\]
Since $\supp (u_0 - u_{m, 0})$ is away from all of $\GG_k$, we can apply Proposition \ref{prop:trace_adv_loc} to obtain 
\[
\| u_0 - u_{m ,0} \|_{W^p_{\adv, \tr}(\GO)} \leq C \| u_0 - u_{m, 0} \|_{W^p_\adv(\GO)} \leq C \| u - u_m \|_{W^p_\adv(\GO)} \to 0
\]
as $m \to \infty$. Thus, it suffices to show that $\| u_k - u_{m, k} \|_{W^p_{\adv, \tr}(\GO)} \to 0$ as $m \to \infty$ for each $k = K_1 + 1, \ldots, K$. In particular, since $\| u_k - u_{m, k} \|_{W^p_\adv(\GO)} \leq C \| u - u_m \|_{W^p_\adv(\GO)}$, it suffices to show that $\| \tilde{\Gg}_\pm u_k - u_{m ,k} \|_{L^p(\GG_\pm; |\weight|)} \to 0$ as $m \to \infty$.

Recalling the definition of $\Gf_k$, we see that $(u_k - u_{m ,k})|_{\p N_k \cap \GO} = 0$ and $\| \tilde{\Gg}_\pm u_k - u_{m ,k} \|_{L^p(\GG_\pm; |\weight|)} = \| \tilde{\Gg}_\pm u_k - u_{m ,k} \|_{L^p(\GG_{\pm, k}; |\weight|)}$. Thus, instead of showing the convergence, we introduce the following estimate: $\| v \|_{L^p(\GG_{\pm, k})} \leq C \| v \|_{W^p_{\adv}(\GO_k)}$ for all $v \in \Lip(\ol{\GO_k})$ with $v|_{\p N_k \cap \GO} = 0$.

We decompose the function $v$ into two parts: $v = v_1 + v_2$, where $v_1 := \Gc_k v$ and $v_2 := (1 - \Gc_k) v$. By Proposition \ref{prop:trace_adv_loc} again, we have $\| v_2 \|_{L^p(\GG_{\pm, k}; |\weight|)} \leq C \| v_2 \|_{W^p_\adv(\GO_k)}$ with some positive constant $C$. In order to derive an estimate for $\| v_1 \|_{L^p(\GG_{\pm, k}; |\weight|)}$, we introduce the following subdomains. Let $\GO_{k, 2, \pm}$ be subsets of $\GO \cap M_{k, 2}$ such that the characteristic line $x(t)$ starting from $x \in \GO \cap M_{k, 2}$ hits $\GG_{\pm, k} \cap M_{k, 2}$ within $t = \pm T/3$. Due to the condition (ii), we have $\GO_{k, 2, +} \cap \GO_{k, 2, -} = \emptyset$. We perform integration by parts on $\GO_{k, 2, \pm}$ to obtain
\[
\int_{\GO_{k, 2, \pm}} (\adv \cdot \nabla v_1) |v_1|^{p-2} v_1\,dx = \frac{1}{p} \int_{\p (\GO_{k, 2, \pm})} |v_1|^p \weight\,d\Gs_x - \frac{1}{p} \int_{\GO_{k, 2, \pm}} (\div \adv) |v_1|^p\,dx.
\]
We notice that $\p \GO_{k, 2, \pm} = (\GG_{\pm, k} \cap M_{k, 2}) \cup (\GO \cap \p M_{k, 2})$, and $v_1 = 0$ on $\GO \cap \p M_{k, 2}$. Thus, invoking $v_1 = 0$ on $\GG_{\pm, k} \setminus \GG_{\pm, k, 2}$, we obtain
\[
\int_{\p (\GO_{k, 2, \pm})} |v_1|^p \weight\,d\Gs_x = \int_{\GG_{\pm, k, 2}} |v_1|^p \weight\,d\Gs_x = \int_{\GG_{\pm, k}} |v_1|^p \weight\,d\Gs_x.
\]
Thus, following the argument in the proof of Lemma \ref{lem:trace_adv}, we get
\[
\| v_1 \|_{L^p(\GG_{\pm, k}; |\weight|)} \leq C_{2, p} \| v_1 \|_{W^p_\adv(\GO_{k, 2, \pm})} \leq C_{2, p} \| v_1 \|_{W^p_\adv(\GO_k)}.
\]
Therefore, we have
\begin{align*}
\| v \|_{L^p(\GG_{\pm, k}; |\weight|)} \leq& \| v_1 \|_{L^p(\GG_{\pm, k}; |\weight|)} + \| v_2 \|_{L^p(\GG_{\pm, k}; |\weight|)}\\
\leq& C (\| v_1 \|_{W^p_\adv(\GO_k)} + \| v_2 \|_{W^p_\adv(\GO_k)} )\\
\leq& C \| v \|_{W^p_\adv(\GO_k)}
\end{align*}
for all $v \in \Lip(\ol{\GO_k})$ with $v|_{\p N_k \cap \GO} = 0$. Since $\Lip(\ol{\GO})$ is dense in $W^p_\adv(\GO)$, the above estimate holds for $v = u_k - u_{m, k}$, which means that
\[
\| u_k - u_{m, k} \|_{W^p_{\adv, \tr}(\GO)} \leq C \| u_k - u_{m, k} \|_{W^p_\adv(\GO)} \to 0
\]
as $m \to \infty$ for each $k = K_1 + 1, \ldots, K$.

Thus, the relation \eqref{eq:density} is proved.

We provide two examples for the above situation in order to facilitate readers' understandings. 

We consider the setting in Example \ref{ex:unbounded} again. In this setting, we have $\ol{\GG_+} \cap \ol{\GG_-} = \{ O \}$. Since the intersection point $\GG_1 = O$ is unique, we take $N_1 = \mathbb{R}^2$ for the sake of simplicity. We shall show that it satisfies the condition (i). Indeed, the boundaries $\GG_\pm$ are parametrized as
\[
\GG_\pm = \{ (\pm s, s) \mid 0 < s < 1 \}.
\]
For fixed $T > 0$, let
\[
\GG_-' := \left\{ (- s, s) \mid 0 \leq s \leq \min \left\{ \frac{T}{4}, \frac{1}{8} \right\} \right\}.
\]
Then, it satisfies $\GG_1 \subset \GG_-' \subset \GG_-$ and $|\GG_-'| > 0$. Moreover, for $x_0 = (-s, s) \in \GG_-'$, the solution $x(t)$ to the Cauchy problem \eqref{Cauchy_boundary} is given by $x(t) = (-s + t, s)$. Thus, it reaches at $\GG_+$ when $t = 2s \leq T/2 < T$. Therefore, the intersection point $O$ satisfies the condition (i).

We give an example for the condition (ii) as follows.

\begin{example} \label{ex:counter}

Let $\adv(x) = (1, 0)$ and let $\GO$ be a bounded domain in $\Rbb^2$ which is surrounded by the following 7 line segments: 
\begin{align*}
\GG_1 :=& \{ (t, t) \mid 0 < t < 1 \},\\
\GG_2 :=& \{ (t, 1) \mid 1 \leq t \leq 2 \},\\
\GG_3 :=& \{ (t, 3-t) \mid 2 < t < 5/2 \},\\
\GG_4 :=& \{ (t, t-2) \mid 5/2 < t < 3 \},\\
\GG_5 :=& \{ (t, 1) \mid 3 \leq t \leq 4 \},\\
\GG_6 :=& \{ (t, 5- t) \mid 4 < t < 5 \},\\
\GG_7 :=& \{ (t, 0) \mid 0 \leq t \leq 5\}.
\end{align*}

\begin{figure}[h]

\centering

\begin{tikzpicture}[scale=1.5]
	\draw[thick, ->] (-1, 0) --(6, 0);
	\draw[thick, ->] (0, -1) --(0, 2);
	\draw[thick, ->] (1.5, 1.5) --(3.5, 1.5);
    \draw[ultra thick] (0,0) --(1,1) --(2,1) --(5/2, 1/2) --(3, 1) --(4, 1) -- (5, 0) --(0, 0);
    \draw[dashed] (0,1) --(1, 1);
    \draw[dashed] (0,0.5) --(2.5,0.5);
    \draw[dashed] (1, 0) --(1, 1); 
    \draw[dashed] (2, 0) --(2, 1); 
    \draw[dashed] (2.5, 0) --(2.5, 0.5);
    \draw[dashed] (3, 0) --(3, 1);
	 \draw[dashed] (4, 0) --(4, 1);   
    \node at (-0.2,1) {1};
    \node at (-0.2, 0.5) {$\frac{1}{2}$};
    \node at (1,-0.2) {1};
    \node at (2, -0.2) {2};
    \node at (2.5, -0.2) {$\frac{5}{2}$};
    \node at (3, -0.2) {3};
    \node at (4, -0.2) {4};
    \node at (5, -0.2) {5};
    \node at (-0.2,-0.2) {0};
    \node [below] at (6,0) {$x$};
    \node [left] at (0,2) {$y$};
    \node at (0.5, 0.8) {$\GG_1$};
    \node at (1.5, 1.2) {$\GG_2$};
    \node at (2.2, 0.6) {$\GG_3$};
    \node at (2.85, 0.6) {$\GG_4$};
    \node at (3.5, 1.2) {$\GG_5$};
    \node at (4.5, 0.8) {$\GG_6$};
    \node at (0.5, -0.2) {$\GG_7$};
    \node at (2.5, 1.7) {$\adv$};
\end{tikzpicture}

\caption{The domain $\GO$ in Example \ref{ex:counter}}

\end{figure}

We notice that $\GG_- = \GG_1 \cup \GG_4$, $\GG_+ = \GG_3 \cup \GG_6$, and $\GG_0 = \GG_2 \cup \GG_4 \cup \GG_7 \cup \{ x_* \}$, where $x_* := (5/2, 1/2) \in \ol{\GG_+} \cap \ol{\GG_-}$. We can see that the solution $x(t)$ to the Cauchy problem \eqref{Cauchy_boundary} starting from $\GG_4 \cup \{ x_* \}$ cannot reach at $\GG_+$ before $t = 1/2$. Thus, the intersection point $x_*$ satisfies the condition (ii). We remark that, in this setting, we have $W^p_{\adv, \tr}(\GO) = W^p_\adv(\GO)$.
\end{example}

\section*{Acknowledgement}
The first author was supported by JST Grant Number JPMJFS2123. The second author was supported by JSPS KAKENHI Grant Number JP21H00999. 


\end{document}